\documentclass[runningheads, envcountsame]{llncs}
\usepackage{graphicx}
\usepackage[utf8]{inputenc}
\usepackage[all,2cell]{xy}

\usepackage{graphicx}
\usepackage{amsmath}
\usepackage{amsfonts}
\usepackage{amssymb}
\usepackage{appendix}

\newcommand{\mC}{\mathsf{C}}

\newcommand{\fibration}[3]{#2 \colon #1 \longrightarrow #3}

\newcommand{\arrow}[3]{#1 \xrightarrow{#2} #3}

\def\mC{\mathsf{C}}

\def\op{\operatorname{op}}

\def\pr{\pi}

\newcommand{\dial}[1]{\mathfrak{Dial}(#1)}

\newcommand{\bemph}[1]{\textbf{\emph{#1}}}
\newcommand{\angbr}[2]{\langle #1,#2 \rangle}

\def\pos{\operatorname{\mathbf{Pos}}}
\def\hey{\operatorname{\mathbf{Hey}}}

\def\set{\operatorname{\mathbf{Set}}}

\newcommand{\quadratocomm}[8]{ \xymatrix@+1pc{ 
#1 \ar[r]^{#5} \ar[d]_{#6} & #2 \ar[d]^{#7} \\
#3 \ar[r]_{#8} & #4 
}}

\renewcommand{\prod}{\forall}
\newcommand{\doctrine}[2]{\xymatrix{#2 \colon #1^{\op}  \ar[r] & \pos }}
\newcommand{\hyperdoctrine}[2]{\xymatrix{#2 \colon #1^{\op}  \ar[r] & \hey}}

\usepackage{hyperref}
\usepackage{enumitem}
\usepackage{csquotes}
\begin{document}
\title{Dialectica Logical Principles\thanks{Research supported by the project MIUR PRIN 2017FTXR IT-MaTTerS (Trotta),
 by a School of Mathematics EPSRC Doctoral Studentship (Spadetto),
and by AFOSR grant FA9550-20-10348 (de Paiva).}}
%
%
\author{Davide Trotta\inst{1}\orcidID{0000-0003-4509-594X} \and
Matteo Spadetto\inst{2}\orcidID{0000-0002-6495-7405} \and
Valeria de Paiva\inst{3}\orcidID{0000-0002-1078-6970}}
\authorrunning{D. Trotta et al.}
%
\institute{University of Pisa, Pisa, Italy \\
\email{trottadavide92@gmail.com}\\
\and University of Leeds, UK \\ \email{matteo.spadetto.42@gmail.com} \and
Topos Institute, Berkeley, USA\\
\email{valeria@topos.institute}}
\maketitle              
\begin{abstract}
G\"odel' s Dialectica interpretation was designed to obtain a relative consistency proof for Heyting arithmetic, to be used in conjunction with the double negation interpretation to obtain the consistency of Peano arithmetic. In recent years, proof theoretic transformations (so-called proof interpretations) that are based  G\"odel’s Dialectica interpretation have been used systematically to extract new content from proofs and so the interpretation has found relevant applications in several areas of mathematics and computer science.
Following our previous work on `G\"odel fibrations', we present a (hyper)doctrine characterisation of the Dialectica which corresponds exactly to the logical description of the interpretation. To show that we derive in the category theory the soundness of the interpretation of the implication connective, as expounded on by Spector and Troelstra. This requires extra logical principles, going beyond intuitionistic logic, Markov's Principle (MP) and the Independence of Premise (IP) principle, as well as some choice. We show how these principles are satisfied in the categorical setting, establishing a tight (internal language) correspondence between the logical system and the categorical framework. This tight correspondence should come handy not only when discussing the applications of the Dialectica already known, like its use to extract computational content from (some) classical theorems (proof mining), its use to help to model specific abstract machines, etc. but also to help devise new applications.

\end{abstract}
\section{Introduction}
Categorical logic is the branch of mathematics in which tools and concepts from category theory are applied to the study of mathematical logic and its connections to theoretical computer science. In broad terms, categorical logic represents both syntax and semantics by a category, and an interpretation by a functor. 
The categorical framework provides a rich conceptual background for logical and type-theoretic constructions. In many cases, the categorical semantics of a logic provides a basis for establishing a correspondence between theories in the logic and instances of an appropriate kind of category. A classic example is the correspondence between theories of $\beta\eta$-equational logic over simply typed lambda calculus and Cartesian closed categories. Categories arising from theories via term-model constructions can usually be characterised up to equivalence by a suitable universal property. This has enabled proofs of meta-theoretical properties of logics by means of an appropriate categorical algebra. One defines a suitable internal language naming relevant constituents of a category, and then applies categorical semantics to turn assertions in a logic over the internal language into corresponding categorical statements. The goal is to obtain `internal language theorems' that allow us to pass freely from the logic/type theory to the categorical universe, in such a way that we can solve issues in whichever framework is more appropriate.

Several kinds of categorical universe are available. Our previous joint work on G\"odel's Dialectica Interpretation \cite{goedel1986} used the fibrational framework expounded by Jacobs in \cite{Jacobs1999}. The identification of syntax-free notions
of quantifier-free formulae using categorical concepts is the key insight to our results in \cite{trotta_et_al:LIPIcs.MFCS.2021.87}. This identification, besides explaining how G\"odel's Dialectica interpretation works as a double completion under products and coproducts, is itself of independent interest, as it deepens our ability to think about first-order logic, using categorical notions. 
Here we show that the notions introduced in our previous paper correspond to well-known (non-intuitionistic but)
constructive principles underlying G\"odel's Dialectica interpretation.

\section{Logical principles in the Dialectica interpretation}\label{section introducing dialectica, MP and IP}

G\"odel's Dialectica interpretation \cite{godel58,goedel1986} 
associates to each formula $\phi$ in the language of arithmetic its  \emph{Dialectica interpretation} $\phi^D$, a formula of the form: 
\[\phi^D =\exists u.\forall x .\phi_D\] which tries to be \emph{as constructive as possible}. 
The most complicated clause of the translation (and, in G\"odel’s words, “the most important one”) is the definition of the translation of the implication connective $(\psi\rightarrow \phi)^D$. This involves two logical principles which are usually not acceptable from an intuitionistic point of view, namely a form of the \emph{Principle of Independence of Premise} (IP) and a generalisation of \emph{Markov's Principle} (MP).
The interpretation is given by:
\[(\psi\rightarrow \phi)^D=\exists V,X.\forall u,y. (\psi_D(u,X(u,y))\rightarrow \phi_D (V(u),y)).\]
The motivation provided 
in the collected works of G\"odel 
for this translation
is that given a witness $u$ for the hypothesis $\psi_D$ one should be able to obtain a witness for the conclusion $\phi_D$, i.e. there exists a function $V$ assigning a witness $V(u)$ of $\phi_D$ to every witness $u$ of $\psi_D$. Moreover, this assignment has to be such that from a counterexample $y$ of the conclusion $\phi_D$ we should be able to find a counterexample $X(u,y)$ to the hypothesis $\psi_D$.  This transformation of counterexamples of the conclusion into counterexamples for the hypothesis is what gives Dialectica its essential character.

We first recall the technical details behind the translation of $(\psi\rightarrow \phi)^D$ (\cite{goedel1986}) showing the precise points in which we have to employ the non-intuitionistic principles (MP) and (IP).
First notice that $\psi^D\rightarrow \phi^D$, that is:
\begin{equation}\label{eq 1 implication}
    \exists u.\forall x.\psi_D(u,x)\rightarrow \exists v.\forall y. \phi_D(v,y)
\end{equation}
is \emph{classically} equivalent to:
\begin{equation}\label{eq 2 implication}
    \forall u.( \forall x.\psi_D(u,x)\rightarrow \exists v.\forall y. \phi_D(v,y)).
\end{equation}
If we apply a special case of the \textbf{Principle of Independence of Premise}, namely:
\begin{align*}
\tag{IP*}
( \forall x.\theta(x)\rightarrow \exists v.\forall y .\eta(v,y))\rightarrow \exists v.( \forall x.\theta(x)\rightarrow \forall y. \eta(v,y))\label{eq 3 implication}
\end{align*}
we obtain that \eqref{eq 2 implication} is equivalent to:
\begin{equation}\label{eq 3 implication}
    \forall u.\exists v.( \forall x.\psi_D(u,x)\rightarrow \forall y. \phi_D(v,y)).
\end{equation}
Moreover, we can see that this is equivalent to:
\begin{equation}\label{eq 4 implication}
    \forall u.\exists v.\forall y. ( \forall x.\psi_D(u,x)\rightarrow  \phi_D(v,y)).
\end{equation}
The next equivalence is motivated by a generalisation of \textbf{Markov’s Principle}, namely:
\begin{equation*}
\tag{MP}
    \neg \forall x. \theta(u,x) \rightarrow \exists  x. \neg \theta(u,x).
\end{equation*}
By applying (MP) we obtain that \eqref{eq 4 implication} is equivalent to:
\begin{equation}\label{eq 5 implication}
    \forall u.\exists v.\forall y.\exists x. ( \psi_D(u,x)\rightarrow  \phi_D(v,y)).
\end{equation}
To conclude that $\psi^D\rightarrow \phi^D=(\psi\rightarrow \phi)^D$ we have to apply the \textbf{Axiom of Choice} (or \textbf{Skolemisation}), i.e.:
 \begin{equation*}
 \tag{AC}
     \forall y. \exists x. \theta(y,x) \rightarrow \exists V. \forall y. \theta(y,V(y))
 \end{equation*}
 twice, obtaining that \eqref{eq 5 implication} is equivalent to:
 \begin{equation*}
   \exists V,X.\forall u,y. (\psi_D(u,X(u,y))\rightarrow \phi_D (V(u),y)).
 \end{equation*}
This analysis (from G\"odel's Collected Works, page 231) highlights the key role the principles (IP), (MP) and (AC)
play in the Dialectica interpretation of implicational formulae. 
The role of the axiom of choice (AC)
has been discussed from  a categorical perspective both by Hofstra \cite{hofstra2011} and in our previous work \cite{trotta_et_al:LIPIcs.MFCS.2021.87}.
We examine the two principles (IP) and (MP) in the next subsections.

\subsection{Independence of Premise}
In logic and proof theory, the \textbf{Principle of Independence of Premise} states that:
\begin{equation*}
       ( \theta\rightarrow \exists u.\eta(u))\rightarrow \exists u.( \theta\rightarrow \eta(u))
\end{equation*}
where $u$ is not a free variable of $\theta$. While this principle is valid in classical logic (it follows from the law of the excluded middle), it does not hold in intuitionistic logic, and it is not generally accepted constructively \cite{AvigadFeferma70}.
The reason why the principle (IP) is not generally accepted constructively is that, from a constructive perspective,  turning any proof of the premise $\phi$ into a proof of $\exists u .\eta(u)$ means turning a proof of $\theta$ into a proof of $\eta (t) $ where $t$ is a witness for the existential quantifier depending on the proof of $\theta$.  In particular, the choice of the witness \emph{depends} on the proof of the premise $\theta$, while the (IP) principle tell us, constructively, that the witness can be chosen independently of any proof of the premise $\theta$.

In the Dialectica translation we only need a particular version of (IP) principle:
\begin{align*}
\tag{IP*}
( \forall y.\theta(y)\rightarrow \exists u.\forall v. \eta(u,v))\rightarrow \exists u.( \forall y.\theta(y)\rightarrow \forall v. \eta(u,v))
\end{align*}
which means that we are asking (IP) to hold not for every formula, but only for those formulas of the form $\forall y .\theta(y)$ with $\theta$ quantifier-free. 
We  recall a useful generalisation of the (IP*) principle, namely: 
\begin{align*}
\tag{IP}
( \theta \rightarrow \exists u . \eta(u))\rightarrow \exists u.( \theta \rightarrow   \eta(u))
\end{align*}
where $\theta$ is $\exists$-free, i.e. $\theta$  contains neither existential quantifiers nor disjunctions (of course, it is also assumed that $u$ is not a free variable of $\theta$). Therefore, the condition that IP holds for every formula of the form $\forall y .\theta (y)$ with $\theta (y)$ quantifier-free is replaced by asking that it holds for every formula \emph{free from the existential quantifier}.

This formulation of (IP) is introduced in \cite{Rathjen2019} where, starting from the observation that intuitionistic finite-type arithmetic is closed under the independence of premise rule for $\exists$-free formula (IPR), it is proved that a similar result holds for  many set theories including Constructive Zermelo-Fraenkel Set Theory (CZF) and Intuitionistic Zermelo-Fraenkel Set Theory (IZF).
The \textbf{Independence of Premise Rule} for $\exists$-free formula (IPR) that we use in this paper, as in \cite{Rathjen2019}, states that:
\begin{align*}
    \tag{IPR}
    \mbox{if } \vdash \theta \rightarrow \exists u . \eta(u)) \mbox{ then } \vdash \exists u.( \theta \rightarrow   \eta(u))
\end{align*}
where $\theta$ is $\exists$-free.

\subsection{Markov's Principle}
\textbf{Markov's Principle} is a statement that originated in the Russian school of constructive mathematics.
Formally, Markov’s principle is usually presented as the statement:
\begin{equation*}
  \neg \neg \exists x.\phi (x) \rightarrow \exists x. \phi (x )
\end{equation*}
where $\phi$ is a quantifier-free formula. Thus, MP in the Dialectica interpretation, namely:
\begin{equation*}
    \tag{MP}
    \neg \forall x.\phi (x)\rightarrow \exists x. \neg \phi (x)
\end{equation*}
with $\phi(x)$ a quantifier-free formula, can be thought of as a generalisation of the Markov Principle above. As  remarked in \cite{AvigadFeferma70}, the reason why  MP is not generally accepted in constructive mathematics is that in general there is no reasonable way to choose constructively a witness $x$ for $\neg \phi (x)$ from a proof that $\forall x. \phi (x)$ leads to a contradiction.
However, in the context of Heyting Arithmetic, i.e. when $x$ ranges over the natural numbers, one can prove that these two formulations of Markov's Principle are equivalent. More details about the computational interpretation of Markov's Principle can be found in \cite{manighetti2016computational}. 
%
We  recall the version of \textbf{Markov’s Rule} (MR) corresponding to Markov's Principle:
\begin{align*}
\tag{MR}
    \mbox{if }\vdash \neg\forall x.\phi (x) \mbox{ then } \vdash \exists x. \neg \phi (x)
\end{align*}
where $\phi(x)$ is a quantifier-free formula.


\section{Logical Doctrines}

One of the most relevant notions of categorical logic which enabled the study of logic from a pure algebraic perspective is that of a \emph{hyperdoctrine},  introduced in a series of seminal papers by F.W. Lawvere to synthesise the structural properties of logical systems \cite{lawvere1969,lawvere1969b,lawvere1970}.
Lawvere’s crucial intuition was to consider logical languages and theories as fibrations to study their 2-categorical properties, e.g. connectives, quantifiers and equality are determined by structural adjunctions. Recall from \cite{lawvere1969} that a \emph{hyperdoctrine} is a functor:
\[\fibration{\mathcal{C}^{\op}}{P}{\mathbf{Hey}}\]
from a cartesian closed category $\mathcal{C}$ to the category of Heyting algebras ${\mathbf{Hey}}$ satisfying some further conditions: for every arrow $\arrow{A}{f}{B}$ in $\mathcal{C}$, the homomorphism $\fibration{P(B)}{P_f}{P(A)}$ of Heyting algebras, where $P_f$ denotes the action of the functor $P$ on the arrow $f$, has a left adjoint $\exists_f$ and a right adjoint $\forall_f$ satisfying the
Beck-Chevalley conditions. The intuition is that a hyperdoctrine determines an appropriate categorical structure to abstract both notions of first order theory and of interpretation.

Semantically, a hyperdoctrine is essentially a generalisation of the contravariant \emph{power-set functor} on the category of sets: $$\fibration{\set^{\op}}{\mathcal{P}}{\mathbf{Hey}}$$ sending any set-theoretic arrow $\arrow{A}{f}{B}$ to the inverse image functor $\arrow{\mathcal{P}B}{\mathcal{P}f=f^{-1}}{\mathcal{P}A}$. 
%
However, from the syntactic point of view, a hyperdoctrine can be seen as the generalisation of the so-called \emph{Lindenbaum-Tarski algebra} of well-formed formulae of a first order theory. In particular, given a first order theory $\mathcal{T}$ in a first order language $\mathcal{L}$, one can consider the functor:
$$\fibration{\mathcal{V}^{\op}}{\mathcal{LT}}{\mathbf{Hey}}$$ 
whose base category $\mathcal{V}$ is the \emph{syntactic} category of $\mathcal{L}$, i.e. the objects of $\mathcal{V}$ are finite  lists $\overrightarrow{x}:=(x_1,\dots,x_n)$ of variables and morphisms are lists of substitutions, while the elements of  $\mathcal{LT}(\overrightarrow{x})$ are given by equivalence classes (with respect to provable reciprocal
consequence $\dashv\vdash $) of well-formed formulae in the context $\overrightarrow{x}$, and order is given by the provable consequences with respect to the fixed theory $\mathcal{T}$. Notice that in this case an existential left adjoint
to the weakening functor $\mathcal{LT}_{\pr}$ is computed by quantifying existentially the variables that are not
involved in the substitution given by the projection (by duality the right adjoint is computed by  quantifying universally).

Recently, several generalisations of the notion of a Lawvere hyperdoctrine were considered, and we refer for example to \cite{maiettipasqualirosolini,maiettirosolini13a,maiettirosolini13b} or to  \cite{pitts02,hyland89} for higher-order versions. 
In this work we consider a natural generalisation of the notion of hyperdoctrine, and we call it simply a \emph{doctrine}.
A \textbf{doctrine} is just a functor:
$$\fibration{\mathcal{C}^{\op}}{P}{\pos}$$ 
where the category $\mathcal{C}$ has finite products and $\pos$ is the category of posets.

Depending on the categorical properties enjoyed by $P$, we get  $P$ to model the corresponding fragments of first order logic formally in a way identical to the one for $\mathcal{P}$, which we call a \textit{generalised Tarski semantics} and which continues to be complete. 
%
%
Again, the syntactic intuition behind the notion of doctrine $\fibration{\mathcal{C}^{\op}}{P}{\pos}$ remains the same, 
 one should think of $\mathcal{C}$ as the category of contexts associated to a given type theory. Given such a context $A$, the elements and the order relation of the posets $P(A)$ represent the predicates in context $A$ and the relation of syntactic provability (with respect to the fragment of first order logic modelled by $P$). Arrows $\arrow{B}{f}{A}$ of $\mathcal{C}$ represent (finite lists of) terms-in-context: 
 $$b: B \; |\;f(b) : A$$ 
 in such a way that the functor $P_f$ models the substitution by the (finite list of) term(s) $f$. For instance, if $\alpha \in PA$ represents a formula in context $a: A\; |\; \alpha (a)$, then $P_f(\alpha) \in P(B)$ represents the formula $b: B\; |\; \alpha(f(b))$ in context $B$ obtained by substituting $f$ into $\alpha$.

Now we recall from \cite{maiettipasqualirosolini,maiettirosolini13a,trotta2020} the notions of existential and universal doctrines.

\begin{definition}\label{def existential doctrine}
A doctrine $\fibration{\mathcal{C}^{\op}}{P}{\pos}$ is \bemph{existential} (resp. \bemph{universal}) if, for every $A_1$ and $A_2$ in $\mathcal{C}$ and every projection $\arrow{A_1\times A_2}{{\pr_i}}{A_i}$, $i=1,2$, the functor:
$$ \arrow{PA_i}{{P_{\pr_i}}}{P(A_1\times A_2)}$$
has a left adjoint $\exists_{\pr_i}$ (resp. a right adjoint $\forall_{\pr_i}$), and these satisfy the \bemph{Beck-Chevalley condition}: for any pullback diagram:
$$
\quadratocomm{X'}{A'}{X}{A}{{\pr'}}{f'}{f}{{\pr}}
$$
with $\pr$ and $\pr'$ projections, for any $\beta$ in $P(X)$ the equality:
$$\exists_{\pr'}P_{f'}\beta= P_f \exists_{\pr}\beta \textnormal{    (  resp. }\forall_{\pr'}P_{f'}\beta= P_f \forall_{\pr}\beta\textnormal{  )}$$
holds (however, observe that the inequality $\exists_{\pr'}P_{f'}\beta\leq P_f \exists_{\pr}\beta \textnormal{    (  resp. }\forall_{\pr'}P_{f'}\beta\geq P_f \forall_{\pr}\beta\textnormal{  )}$ always holds).
\end{definition}

If a doctrine $\fibration{\mathcal{C}^{\op}}{P}{\pos}$ is existential and $\alpha   \in P(A\times B)$ is a formula-in-context $a: A,b: B \; |\; \alpha (a,b)$ and $\arrow{A \times B}{\pi_A}{A}$ is the product projection on the component $A$, then $\exists_{\pi_A}\alpha \in PA$ represents the formula $a: A\;|\; \exists b: B . \alpha(a,b)$ in context $A$. Analogously, if the doctrine $P$ is universal, then $\forall_{\pi_A}\alpha \in PA$ represents the formula $a: A \; |\; \forall b: B .\alpha(a,b)$ in context $A$. 
This interpretation is sound and complete for the usual reasons: this is how classic Tarski semantics can be characterised in terms of categorical properties of the powerset functor $\fibration{\set^{\op}}{\mathcal{P}}{\pos}$.


We recall how 
we think of the base category of a given doctrine as the category of contexts of a given type  theory and the elements of the fibre of a given context as the predicates in that context. This intuition provides a categorical equivalence between logical theories and doctrines, via the so-called \emph{internal language} of a doctrine. 
The internal language of a doctrine $P$ essentially constitutes a syntax endowed with a semantics induced by $P$ itself: 
there is a way to interpret every sequent in the fragment of first-order logic modelled by $P$ into a categorical statement involving $P$. This interpretation is sound and complete; this is precisely why we can deduce properties of $P$ through a purely syntactical procedure.
We define the following notation for this syntax, taking advantage of these equivalent ways of reasoning about doctrines and logic.


\noindent
\textbf{Notation.} From now on, we shall employ the logical language provided by the \emph{internal language} of a doctrine and write:
\[ a_1: A_1,\dots,a_n:A_n \; | \; \phi(a_1,\dots,a_n)\vdash \psi(a_1,\dots,a_n) \]
instead of:
\[\phi\leq \psi\]
in the fibre $P(A_1\times \cdots\times A_n)$. Similarly, we write:
\[a:A\; |\; \phi(a)\vdash \exists b : B .\psi (a,b) \text{ and } a:A\; |\; \phi(a)\vdash \forall b : B . \psi (a,b) \]
in place of:
\[\phi\leq \exists_{\pr_A} \psi \text{ and } \phi\leq \forall_{\pr_A} \psi \]
in the fibre $P(A)$. Also, we write $a:A \; |\; \phi \dashv\vdash \psi$ to abbreviate $a:A \; |\; \phi \vdash \psi$ and $a:A \; |\; \psi \vdash \phi$. Substitutions via given terms (i.e. reindexings and weakenings) are modelled by pulling back along those given terms. Applications of propositional connectives are interpreted by using the corresponding operations in the fibres of the given doctrine. Finally, when the type of a quantified variable is clear from the context, we will omit the type for the  sake of readability.

\section{Logical principles via universal properties}
It is possible to characterise, in terms of weak universal properties, those predicates of a doctrine that are free from a quantifier. In the following definitions, we pursue this idea of defining those elements of an existential doctrine $\fibration{\mathcal{C}^{\op}}{P}{\pos}$ which are \emph{free from the left adjoints} $\exists_{\pr}$. This idea was originally introduced in \cite{trottamaietti2020}, and then further developed and generalised in the fibrational setting in \cite{trotta_et_al:LIPIcs.MFCS.2021.87}.

\begin{definition}\label{definition (P,lambda)-atomic0}
Let $\fibration{\mathcal{C}^{\op}}{P}{\pos}$ be an existential doctrine and let $A$ be an object of $\mathcal{C}$. A predicate $\alpha$ of the fibre $P(A)$ is said to be an  \bemph{existential splitting} if it satisfies the following weak universal property: for every projection $\arrow{A\times B}{\pr_A}{A}$ of $\mathcal{C}$ and every predicate $\beta\in P(A\times B)$ such that $\alpha\leq \exists_{\pr_A}(\beta)$, there exists an arrow $\arrow{A}{g}{B}$ such that:
\[    \alpha \leq P_{\angbr{1_A}{g}}(\beta). \]
\end{definition}
Existential splittings stable under re-indexing are called \emph{existential-free elements}. Thus we introduce the following definition:
\begin{definition}\label{definition (P,lambda)-atomic}
Let $\fibration{\mathcal{C}^{\op}}{P}{\pos}$ be an existential doctrine and let $I$ be an object of $\mathcal{C}$. A predicate $\alpha$ of the fibre $P(I)$ is said to be  \bemph{existential-free} if $P_f(\alpha)$ is an existential splitting for every morphism $\arrow{A}{f}{I}$.
\end{definition}

Employing the presentation of doctrines via internal language,  we require that for the formula $i:I\; |\; \alpha(i)$ to be free from the existential quantifier, whenever $a:A \;|\; \alpha(f(a))\vdash \exists b: B.\beta (a,b)$, for some term $a:A\;|\; f(a): I$, then there is a term $a:A\; |\; g(a):B$ such that $a:A\; |\; \alpha(f(a))\vdash \beta(a,g(a))$.

Observe that in general we always have that $a:A\; |\; \beta(a,g(a))\vdash \exists b: B.\beta (a,b)$, in other words $  P_{\langle 1_A,g\rangle}\beta\leq \exists_{\pi_A}\beta$. In fact, it is the case that $\beta \leq P_{\pi_A}\exists_{\pi_A}\beta$ (as this arrow of $P(A \times B)$ is nothing but the unit of the adjunction $\exists_{\pi_A}\dashv P_{\pi_A}$), hence a re-indexing by the term $\langle 1_A,g\rangle$ yields the desired inequality. Therefore, the property that we are requiring for $i:I\; |\; \alpha(i)$ turns out to be the following: whenever there are proofs of $\exists b: B.\beta(a,b)$ from $\alpha(f(a))$, at least one of them factors through the canonical proof of $\exists b: B .\beta (a,b)$ from $\beta(a,g(a))$ for some term $a:A\; |\; g(a) : B$. This fact implies that, while freely adding the existential quantifiers to a doctrine, we do not add a new sequent $\alpha \vdash \exists b.\beta(b)$ (where $\alpha$ and $\beta(b)$ are predicates in the doctrine we started from) as long as we do not allow a sequent $\alpha \vdash \beta(g)$ as well, for some term $g$ (see \cite{trottapasquale2020} for more details).  For the proof-relevant versions of this definition we refer to \cite{trotta_et_al:LIPIcs.MFCS.2021.87}.

We dualise the previous Definitions \ref{definition (P,lambda)-atomic0} and Definition \ref{definition (P,lambda)-atomic}  to get the corresponding ones for the universal quantifier.

\begin{definition}\label{definition (P,lambda)-coatomic0}
Let $\fibration{\mathcal{C}^{\op}}{P}{\pos}$ be a universal doctrine and let $A$ be an object of $\mathcal{C}$. A predicate $\alpha$ of the fibre $P(A)$ is said to be a  \bemph{universal splitting} if it satisfies the following weak universal property: for every projection $\arrow{A\times B}{\pr_A}{A}$ of $\mathcal{C}$ and every predicate $\beta\in P(A\times B)$ such that $\forall_{\pr_A}(\beta)\leq \alpha$, there exists an arrow $\arrow{A}{g}{B}$ such that:
\[    P_{\angbr{1_A}{g}}(\beta)\leq \alpha. \]
\end{definition}

\begin{definition}\label{definition (P,lambda)-coatomic}
Let $\fibration{\mathcal{C}^{\op}}{P}{\pos}$ be a universal doctrine and let $I$ be an object of $\mathcal{C}$. A predicate $\alpha$ of the fibre $P(I)$ is said to be  \bemph{universal-free} if $P_f(\alpha)$ is a universal splitting for every morphism $\arrow{A}{f}{I}$.
\end{definition}

The property we require of the formula $i:I\; |\; \alpha(i)$, so that it is free from the universal quantifiers, is that whenever $a:A\; | \; \forall b: B .\beta (a,b)\vdash \alpha(f(a))$, for some term $a:A \; |\; f(a):I$, then there is a term $a:A\; |\; g(a):B$  such that $a:A\;| \; \beta(a,g(a))\vdash \alpha(f(a))$.


\begin{definition}
Let $\fibration{\mathcal{C}^{\op}}{P}{\pos}$ be a doctrine. If $P$ is existential, we say that $P$ has \bemph{enough existential-free predicates} if, for every object $I$ of $\mathcal{C}$ and every predicate $\alpha\in PI$, there exist an object $A$ and an existential-free object $\beta$ in $P(I \times A)$ such that
$\alpha=\exists_{\pr_I}\beta.$

Analogously, if $P$ is universal, we say that $P$ has \bemph{enough universal-free predicates} if, for every object $I$ of $\mathcal{C}$ and every predicate $\alpha\in PI$, there exist an object $A$ and a universal-free object $\beta$ in $P(I \times A)$ such that
$\alpha=\forall_{\pr_I}\beta.$
\end{definition}

Now we can introduce a particular kind of doctrine called a \emph{G\"odel doctrine}. This  definition works as a synthesis of our process of categorification of the logical notions.

\begin{definition}\label{goedel}
A doctrine $\fibration{\mathcal{C}^{\op}}{P}{\pos}$ is called a \bemph{G\"odel doctrine} if:
\begin{enumerate}
    \item the category $\mathcal{C}$ is cartesian closed;
    \item the doctrine $P$ is existential and universal;
    \item the doctrine $P$ has enough existential-free predicates;
    \item the existential-free objects of $P$ are stable under universal quantification, i.e. if $\alpha\in P(A)$ is existential-free, then $\forall_{\pr}(\alpha)$ is existential-free for every projection $\pr$ from $A$;
    \item the sub-doctrine $\doctrine{\mathcal{C}}{P'}$ of the existential-free predicates of $P$ has enough universal-free predicates.
\end{enumerate}
\end{definition}

The fourth point 
of the Definition \ref{goedel} above implies that, given a G\"odel doctrine $\fibration{\mathcal{C}^{\op}}{P}{\pos}$, the sub-doctrine $\fibration{\mathcal{C}^{\op}}{P'}{\pos}$, such that $P'(A)$ is the poset of existential-free predicates contained in $P(A)$ for any object $A$ of $\mathcal{C}$, is a universal doctrine. From a purely logical perspective, requiring existential-free elements to be stable under universal quantification is quite natural since this can be also read as \emph{if $\alpha (x)$ is an existential-free predicate, then $\forall x: X. \alpha (x)$ is again an existential-free predicate}.

An element $\alpha$ of a fibre $P(A)$ of a G\"odel doctrine $P$ that is both an existential-free predicate and a universal-free predicate in the sub-doctrine $P'$ of existential-free elements of $P$ is called a \textbf{quantifier-free predicate} of $P$.
In order to simplify the notation, but also to make clear the connection with the logical presentation in the Dialectica interpretation, we will use the notation $\alpha_D$ to indicate an element $\alpha$ which is a quantifier-free predicate. 
Applying the definition of a G\"odel doctrine we obtain the following result.
\begin{theorem}\label{theorem prenex normal form}
Let $\doctrine{\mathcal{C}}{P}$ be a G\"odel doctrine, and let $\alpha$ be an element of $P(A)$. Then there exists a quantifier-free predicate $\alpha_D$ of $P(I\times U\times X)$ such that:
\[i:I\; |\; \alpha (i) \dashv\vdash \exists u: U.\forall x: X. \alpha_D (i,u,x).\]
\end{theorem}
This theorem shows that in a G\"odel doctrine every formula admits a presentation of the precise form used in the Dialectica translation. 

Now we show that employing the properties of a G\"odel doctrine we can provide a complete categorical description and presentation of the chain of equivalences involved in the Dialectica interpretation of the implicational formulae. In particular, we  show that the crucial steps where (IP) and (MP) are applied are represented categorically via the notions of existential-free element and universal-free element. 

Let us consider a G\"odel fibration $\doctrine{\mathcal{C}}{P}$ and two quantifier-free predicates $\psi_D\in P(U\times X)$ and $\phi_D\in P(V\times Y)$.
First notice that the following equivalence follows by definition of left adjoint functor (for sake of readability we omit the types of quantified variables as we anticipated in the previous section):
\begin{align}\label{eq first equivalence Godel doctrine}
    - \; | \; \exists u.\forall x.  \psi_D (u,x)\vdash  \exists v.\forall y.  \phi_D (v,y)  &\iff 
  u: U \; | \; \forall x.  \psi_D (u,x)\vdash  \exists v.\forall y.  \phi_D (v,y) 
\end{align} 
Now we employ the fact that the predicate $ \forall x.  \psi_D (u,x)$ is existential-free in the G\"odel doctrine, obtaining that there exists an arrow $\arrow{U }{f_0}{V}$, such that:
\begin{align*}
     u:U \; | \; \forall x.  \psi_D (u,x)\vdash  \exists v.\forall y . \phi_D (v,y) & \iff
      u:U \; | \;   \forall x  .\psi_D (u,x)\vdash   \forall y . \phi_D (f_0(u),y)
       \end{align*}
Then, since the universal quantifier is right adjoint to the weakening functor, we have that:
\begin{align*}
       u:U \; | \;  \forall x . \psi_D (u,x)\vdash   \forall y . \phi_D (f_0(u),y) & \iff
      u:U , y:Y\; | \;   \forall x . \psi_D (u,x)\vdash    \phi_D (f_0(u),y).
       \end{align*}
Now we employ the fact that  $ \phi_D (f_0(u),y)$ is universal-free in the subdoctrine of existential-free elements of $P$. Notice that since $\psi_D (u,x)$ is a quantifier-free element of the G\"odel doctrine, we have that $\forall x.  \psi_D (u,x)$ is existential free. Recall that this follows from the fact that in every G\"odel doctrine, existential-free elements are stable under universal quantification (this is the last point Definition \ref{goedel}). Therefore we can conclude that there exists an arrow $\arrow{U\times Y}{f_1}{X}$ of $\mathcal{C}$ such that:
\begin{align}\label{eq final equivalence Godel doctrine}
            u:U , y:Y\; | \;  \forall x.  \psi_D (u,x)\vdash    \phi_D (f_0(u),y)  \iff
             u:U , y:Y\; | \;  \psi_D (u,f_1(u,y))\vdash    \phi_D (f_0(u),y)
       \end{align}
Then, combining the equivalence \eqref{eq first equivalence Godel doctrine} and \eqref{eq final equivalence Godel doctrine}, we obtain the following equivalence:
\begin{align*}
   -\;|\;  \exists u.\forall x.  \psi_D (u,x)\vdash  \exists v.\forall y.  \phi_D (v,y)  &\iff \\
    \text{there exist $(f_0,f_1)$ s.t. }
   u:U , y:Y\; | \;  \psi_D (u,f_1(u,y))\vdash    \phi_D (f_0(u),y).
\end{align*} 
The arrow $\arrow{U }{f_0}{V}$ represents the \emph{witness function}, i.e. it assigns to every witness $u$ of the hypothesis a witness $f_0(u)$ of the thesis, while the arrow $\arrow{U\times Y}{f_1}{X}$ represents the \emph{counterexample function}. Notice that while the witness function $f_0(u)$ depends only of the witness $u$ the counterexample function $f_1(u,y)$ depends on a witness of the hypothesis and a counterexample of the thesis. This is a quite natural fact because, considering the constructive point of view, the counterxample has to be relative to a witness validating the thesis.

This provides a proof of the following theorem which establishes the connection between G\"odel doctrines and the Dialectica interpretation. Notice that for the sake of clarity, but also to keep the presentation closer to the original one, in the previous paragraph we have considered formulae $\exists u.\forall x.  \psi_D (u,x)$ with no free-variables. However, the previous arguments can be easily generalised also for the case of formulae of the form $\exists u.\forall x.  \psi_D (u,x,i)$, i.e. with free-variables $i$. In this case one needs to change just the domains of the functions $f_0$ and $f_1$, since they are allowed to depend  also on the free-variables. 
\begin{theorem}\label{theorem principal 1}
Let $\doctrine{\mathcal{C}}{P}$ be a G\"odel doctrine. Then for every $\psi_D\in P(I\times U\times X )$ and $\phi_D\in P(I\times V\times Y)$ quantifier-free predicates of $P$ we have that:
\[ i:I\; |\; \exists u.\forall x . \psi_D (i,u,x)\vdash \exists v.\forall y.  \phi_D (i,v,y)\]
if and only if there exists $\arrow{I\times U}{f_0}{V}$ and $\arrow{I\times U\times Y}{f_1}{X}$ such that:
\[u:U,y:Y,i:I\; |\; \psi_D(i,u,f_1(i,u,y))\vdash \phi_D (i,f_0(i,u),y).\]

\end{theorem}
This theorem shows that the notion of G\"odel doctrine encapsulates in a pure form the basic mathematical feature of the Dialectica interpretation, namely its interpretation of implication, which corresponds to the existence of functionals of types $f_0:U\to V$ and $f_1:U\times Y\to X$ as described.
One should think of this as saying that a proof of $\exists u.\forall x . \psi_D (i,u,x)\rightarrow \exists v.\forall y.  \phi_D (i,v,y)$ is obtained by transforming to $\forall u.\exists v.\forall y. \exists x. (\psi_D (i,u,x)\rightarrow   \phi_D (i,v,y))$, and then Skolemizing along the lines explained in the Section \ref{section introducing dialectica, MP and IP} and by Troelstra \cite{goedel1986}.
So, combining Theorems \ref{theorem prenex normal form} and \ref{theorem principal 1} we have  strong evidence that the notion of G\"odel doctrine really provides a categorical abstraction of the main concepts involved in the Dialectica translation. Now we show that this kind of doctrine embodies also the \emph{logical principles} involved in the translation. The first principle we consider it the axiom of choice (AC) also sometimes called the principle of Skolemisation. Since the following theorem is the proof-irrelevant version of  the proof we refer to \cite[Prop. 2.8]{trotta_et_al:LIPIcs.MFCS.2021.87} for the detailed proof.

\begin{theorem}\label{skolem}
Every G\"odel doctrine $\doctrine{\mathcal{C}}{P}$ validates the \bemph{Skolemisation principle}, that is:
$$ a_1 : A_1 \; | \; \forall a_2. \exists b. \alpha(a_1,a_2,b)\dashv\vdash \exists f. \forall a_2. \alpha (a_1,a_2,fa_2)$$ 
where $ f : B^{A_2}$ and $fa_2$ denote the evaluation of $f$ on $a_2$, whenever $\alpha(a_1,a_2,b)$ is a predicate in the context $A_1 \times A_2 \times B$.
\end{theorem}
\begin{remark}
In the proof of Theorem \ref{skolem} we do not need the property \textit{5.} of Definition \ref{goedel}. That is why, according to \cite{trotta_et_al:LIPIcs.MFCS.2021.87}, one calls a Skolem doctrine a doctrine satisfying all of the properties satisfied by a G\"odel doctrine, except for the \textit{5.} one.
\end{remark}

Recall that the notion of Dialectica category introduced in \cite{dePaiva1989dialectica} has been generalised to the fibrational setting in \cite{hofstra2011}, and then, in particular, we can consider the proof-irrelevant construction associating  a doctrine $\dial{P}$ to a given doctrine $P$:\\

\noindent
\textbf{Dialectica construction.} Let $\doctrine{\mathcal{C}}{P}$ be a doctrine whose base category $\mC$ is cartesian closed. We define the \textbf{dialectica doctrine} $\doctrine{\mathcal{C}}{\dial{P}}$ the functor sending an object $I$ into the poset $\dial{P}(I)$ defined as follows:
\begin{itemize}
\item \textbf{objects} are quadruples $(I, X,U,\alpha)$ where $I,X$ and $U$ are objects of the base category $\mathcal{C}$ and $\alpha\in P(I\times X\times U)$;
\item \textbf{partial order:} we stipulate that $(I, U,X,\alpha)\leq (I,V,Y,\beta)$ if there exists a pair $(f_0,f_1)$, where $\arrow{I\times U}{f_0}{V}$ and  $\arrow{I\times U\times Y}{f_1}{X}$  are morphisms of $\mathcal{C}$  such that: $$\alpha(i,u,f_1(i,u,y))\leq \beta (i,f_0(i,u),y).$$ \end{itemize} In \cite{trotta_et_al:LIPIcs.MFCS.2021.87} we proved that fibration is an instance of the Dialectica construction if and only if it is a G\"odel fibration, and to prove this result we employ the decomposition of the Dialectica monad as free-simple-product completion followed by the free-simple-coproduct completion of fibrations. So we can deduce the same result for the proof-irrelevant version here simply as a particular case.

However, notice that employing Theorems \ref{theorem prenex normal form} and \ref{theorem principal 1} we have another simpler and more direct way for proving such correspondence, because Theorem \ref{theorem principal 1} states that the order defined in the fibres of a G\"odel doctrine is exactly the same order defined in a dialectica doctrine. 
The idea is that if $P$ is a G\"odel doctrine and $P'$ is the subdoctrine of quantifier-free elements of $P$ it is easy to check that the assignment $\arrow{P(I)}{(-)^D}{\dial{P'}(I)}$ sending $\alpha\mapsto (I,X,U,\alpha_D)$ where $\alpha_D$ is the equantifier-free element such that $\alpha(i)\dashv\vdash \exists u\forall x\alpha_D(i,u,x)$  (which exists by by Theorem \ref{theorem prenex normal form}), provides an isomorphism of posets by Theorem \ref{theorem principal 1}, and it can be extended to an isomorphism of existential and universal doctrines. 
\begin{theorem}\label{theorem godel iff dialectica}
Every G\"odel doctrine $P$ is equivalent to the Dialectica completion $\dial{P'}$ of the full subdoctrine $P'$ of $P$ consisting of the quantifier-free predicates of $P$.
\end{theorem}
Therefore, we have that Theorem \ref{theorem godel iff dialectica} provides another way of thinking about Dialectica doctrines (or Dialectica categories) since it underlines the logical properties that a doctrine has to satisfy in order to be an instance of the Dialectica construction.

\section{Logical Principles in G\"odel Hyperdoctrines}
G\"odel doctrines provide a categorical framework that generalises the principal concepts underlying the Dialectica translation, such as the existence of witness and counterexample functions whenever we have an implication $i:I\; |\; \exists u.\forall x . \psi_D (u,x,i)\vdash \exists v.\forall y . \phi_D (v,y,i)$. 
The key idea is that, intuitively, the notion of \emph{existential-quantifier-free} objects can be seen as a
reformulation of the \emph{independence of premises rule}, while \emph{product-quantifier-free} objects can be seen as a reformulation of \emph{Markov's rule}. Notice that in the proof of Theorem \ref{theorem principal 1} existential and universal free elements play the same role that (IP) and (MP) have in the Dialectica interpretation of implicational formulae.

The main goal of this section is to formalise this intuition showing the exact connection between the rules (IPR) and (MR) and G\"odel doctrines. 
So, first of all we have to equip G\"odel doctrines with the appropriate Heyting structure in the fibres in order to be able to formally express these principles. Therefore, we have to consider G\"odel hyperdoctrines.
\begin{definition}
A hyperdoctrine $\hyperdoctrine{\mathcal{C}}{P}$ is said a \bemph{G\"odel hyperdoctrine} when $P$ is a G\"odel doctrine.

\end{definition}

From a logical perspective, one might want the quantifier-free predicates to be closed with respect to all of the propositional connectives (or equivalently that $P$ is the dialectica completion of a hyperdoctrine itself - see \cite{trottapasquale2020}), since this is what happens in logic. However, we do not need such a strong condition here. 
We only require in the next statements that $\bot$ is quantifier-free and/or that $\top$ is existential free.

\begin{theorem}\label{IP in Godel doctrine}
Every G\"odel hyperdoctrine $\hyperdoctrine{\mathcal{C}}{P}$ satisfies the \bemph{Rule of Independence of Premise}, i.e. whenever $\beta \in P(A \times B)$ and $\alpha\in P(A)$ is a existential-free predicate, it is the case that: \[a:A\; | \; \top \vdash \alpha(a) \rightarrow \exists b.\beta(a,b) \mbox{ implies that } a:A\; | \; \top \vdash \exists b.(\alpha(a)\rightarrow \beta(a,b)).\]

\end{theorem}
\begin{proof}
Let us assume that $a:A\; | \; \top \vdash \alpha(a) \rightarrow \exists b.\beta(a,b)$. Then it is the case that $a : A\;|\; \alpha(a) \vdash \exists b.\beta(a,b)$. Since $\alpha(a)$ is free from the existential quantifier, it is the case that there is a term in context $a:A\; | \;t(a) : B$ such that: \[a:A\; | \; \top \vdash \alpha(a) \rightarrow \beta(a,t(a)).\] Therefore, since: \[ a:A\;|\;\alpha(a) \rightarrow \beta(a,t(a))\vdash \exists b.(\alpha(a) \rightarrow \beta(a,b))\] (as this holds for any predicate $\gamma(a,-)$ in place of the predicate $\alpha_D(a)\rightarrow \beta(a,-)$) we conclude that: \[a:A\; | \; \top \vdash \exists b.(\alpha(a)\rightarrow \beta(a,b)).\]
\end{proof}
Notice that Theorem \ref{IP in Godel doctrine} formalises precisely the intuition that the notion of existential-free element can be seen as a reformulation of the independence of premises rule: in a G\"odel hyperdoctrine we have that existential-free elements are \emph{exactly} elements satisfying the independence of premises rule.

\begin{theorem}\label{theorem weak markov rule}
Every G\"odel hyperdoctrine $\hyperdoctrine{\mathcal{C}}{P}$ satisfies the following \bemph{Modified Markov's Rule}, i.e. whenever $\beta_D \in P(A)$ is a quantifier-free predicate and $\alpha \in P(A\times B)$ is an existential-free predicate, it is the case that: 
\[a:A\; | \; \top \vdash (\forall b. \alpha(a,b)) \rightarrow \beta_D(a) \mbox{ implies that } a:A\; | \; \top \vdash \exists b.(\alpha(a,b)\rightarrow \beta_D(a)).\]
\end{theorem}
\begin{proof}
Let us assume that $a:A\; | \; \top \vdash (\forall b. \alpha(a,b)) \rightarrow \beta_D(a) $. Then it is the case that $a:A\; | \;  (\forall b. \alpha(a,b)) \vdash \beta_D(a) $. Hence, since $\beta_D$ is quantifier-free and $\alpha$ is existential-free, there exists a term in context  $a:A\; | \;t(a) : B$ such that:
\[a:A\; | \;  \top\vdash \alpha(a,t(a)) \rightarrow \beta_D(a) \]
therefore, since:
\[ a:A\;|\;\alpha(a,t(a)) \rightarrow \beta(a)\vdash \exists b.(\alpha(a,b) \rightarrow \beta_D(a))\]
we can conclude that:
\[ a:A\; | \; \top \vdash \exists b.(\alpha(a,b)\rightarrow \beta_D(a)).\]
\end{proof}
While for the case of (IPR) we have that existential-free elements of a G\"odel hyperdoctrine correspond to formulae satisfying (IPR), we have that the elements of a G\"odel doctrine that are quantifier-free, i.e. universal-free in the subdoctrine of existential-free elements, are exactly those satisfying a modified Markov's Rule by Theorem \ref{theorem weak markov rule}. Moreover, notice that this Modified Markov's Rule is exactly the one we need in the equivalence between \eqref{eq 4 implication} and \eqref{eq 5 implication} in the interpretation of the implication in Section \ref{section introducing dialectica, MP and IP}. Alternatively, in order to get this equivalence one requires $\beta_D$ to satisfy the law of excluded middle and the usual Markov's Rule (see Corollary \ref{theorem Markov}), as these two assumptions yield the Modified Markov's Rule. In particular, any boolean doctrine (a hyperdoctrine modelling the law of excluded middle) satisfies the Modified Markov's Rule (see Remark \ref{goedel vs boolean}).

To obtain the usual Markov Rule as corollary of Theorem \ref{theorem weak markov rule}, we simply have to require the bottom element $\bot$ of a G\"odel hyperdoctrine to be \emph{quantifier-free}.
\begin{corollary}\label{theorem Markov}

Every G\"odel hyperdoctrine  $\hyperdoctrine{\mathcal{C}}{P}$ such that $\bot$ is a quantifier-free predicate satisfies \bemph{Markov's Rule}, i.e. for every quantifier-free element $\alpha_D\in P(A\times B)$ it is the case that:
\[b:B\; | \; \top \vdash \neg \forall a. \alpha_D (a,b) \mbox{ implies that } b:B\; | \; \top \vdash \exists a . \neg \alpha_D (a,b).\]
\end{corollary}
\begin{proof}
It follows by Theorem \ref{theorem weak markov rule} just by replacing $\beta_D$ with $\bot$, that is quantifier-free by hypothesis.
\end{proof}

\begin{remark}\label{goedel vs boolean}
Any boolean doctrine satisfies the Rule of Independence of Premises and the (Modified) Markov Rule. In general these are not satisfied by a usual hyperdoctrine, because they are not satisfied by intuitionistic first-order logic. It turns out that \textit{the logic modelled by a G\"odel hyperdoctrine is right in-between intuitionistic first-order  and classical first-order logic}: it is  powerful enough to guarantee the equivalences in Section \ref{section introducing dialectica, MP and IP} that justify the Dialectica interpretation of the implication.
\end{remark}

We conclude  by presenting two other results about the Rule of Choice and the Counterexample Property previously defined in \cite{trottapasquale2020}, which  follow directly from the  definitions of existential-free and universal-free elements.


\begin{corollary}\label{proposition Counterexample Property}
Every G\"odel hyperdoctrine  $\hyperdoctrine{\mathcal{C}}{P}$ such that $\bot$ is a quantifier-free object satisfies the \bemph{Counterexample Property}, that is, whenever: $$a : A \; |\; \forall b. \alpha(a,b)\vdash \bot$$ for some predicate $\alpha(a,b)\in P(A\times B)$, then it is the case that: $$a : A \; |\; \alpha(a,g(a))\vdash \bot$$ for some term in context $a : A \;|\; g(a) : B$.
\end{corollary}

\begin{corollary}\label{proposition rule of choice}
Every G\"odel hyperdoctrine $\hyperdoctrine{\mathcal{C}}{P}$ such that $\top$ is existential-free satisfies the \bemph{Rule of Choice}, that is, whenever: $$a : A \; |\; \top \vdash \exists b. \alpha (a,b)$$ for some existential-free predicate $\alpha \in P(A\times B)$, then it is the case that: $$a : A \; |\; \top \vdash \alpha (a,g(a))$$ for some term in context $a : A \;|\; g(a) : B$.
\end{corollary} 

The rule appearing in Corollary \ref{proposition rule of choice} is called \emph{Rule of Choice} in \cite{maiettipasqualirosolini}, while it appears as \emph{explicit definability} in \cite{Rathjen2019}.

\section{Conclusion}
We have recast our previous fibrational based modelling of G\"odel's interpretation\cite{trotta_et_al:LIPIcs.MFCS.2021.87} in terms of categorical (hyper)doctrines.
We show that the notions we considered in our previous work  (existential-free and universal-free objects) really provide a categorical explanation of the traditional syntactic notions as described in \cite{goedel1986}. This means that we are able to mimic completely the purely logical explanation of the interpretation, given by Spector and expounded on by Troelstra \cite{goedel1986}, using categorical notions. We show how to interpret logical implications using the Dialectica transformation. Through this process we explain how we go beyond intuitionistic principles, adopting both the Independence of Premise (IP) principle and Markov's Principle (MP) as well as the axiom of choice in the logic.

Our main results show the perfect correspondence between the logical and the categorical tools, in the cases of Markov's principle (MP) and the independence of premise (IP) principle. This is very interesting by itself, as it shows that the categorical modelling really captures all the essential features of the interpretation. But it also opens new possibilities for modelling of constructive set theories (in the style of Nemoto and Rathjan \cite{Rathjen2019}) and of categorical modelling of intermediate logics (intuitionistic propositional logic plus (IP) or (MK), see \cite{Aschieri2016,Herbelin2010}). This leads into applications both into the investigation of functional abstract machines \cite{Pedrot2014,mossvonglehn2018}, of reverse mathematics \cite{Rathjen2019}  and of quantified modal logic \cite{Shimura1994}.

\subsubsection*{Acknowledgements.} 
We would like to thank Milly Maietti for ideas and discussions that inspired this work.

\bibliographystyle{splncs04}
\bibliography{references}
\end{document}